\documentclass[12pt,a4paper,leqno]{amsart}

\usepackage{bbm}

\usepackage{amssymb,amsmath,amsthm}
\usepackage{graphicx}
\numberwithin{equation}{section}
\usepackage{genyoungtabtikz}
\usepackage{ytableau}

\theoremstyle{thmit} 
\newtheorem{thm}{Theorem}[section]
\newtheorem{lem}[thm]{Lemma}

\usepackage{mathrsfs}

\theoremstyle{definition}
\newtheorem{remark}[thm]{Remark}

\newtheorem{example}[thm]{Example}

\def\zetas{\zeta^{\star}}

\newcommand{\1}{\mathbbm{1}}
\makeatletter
\@namedef{subjclassname@2020}{%
  \textup{2020} Mathematics Subject Classification}
\makeatother

\numberwithin{equation}{section}

\allowdisplaybreaks

\begin{document} 

\title[Schur multiple zeta-functions of Hurwitz type]{
Schur multiple zeta-functions of Hurwitz type
}
\author {Kohji Matsumoto}
\author{Maki Nakasuji}


\thanks{This work was supported by Japan Society for the Promotion of Science, Grant-in-Aid for Scientific Research No. 22K03267 (K. Matsumoto) and No. 22K03274(M. Nakasuji).}

\maketitle

\begin{abstract}
We study the Hurwitz-type analogue of Schur multiple zeta-functions involving shifting
parameters.
We extend various formulas, known for ordinary Schur multiple zeta-functions,
to the case of Hurwitz type.
We also mention unpublished results proved by Yamamoto and by Minoguchi.
Further we present new formulas obtained by performing differentiation with
respect to shifting parameters. 
\end{abstract}

\bigskip
\section{Introduction}\label{sec-1}


In 1882, Hurwitz introduced a shifted version of the Riemann zeta-function
defined by 
$$\zeta(s,x)=\sum_{m=0}^{\infty}\frac{1}{(m+x)^{s}},$$ 
where $x>0$ and $s\in\mathbb{C}$.
This is called the Hurwitz zeta-function.    
The Riemann zeta function is $\zeta(s)=\zeta(s,1)$.
A similar shifted version can be considered for the multiple zeta-function and zeta-star function of Euler-Zagier type. These are called the multiple
Hurwitz zeta(-star)-functions and written by:
\begin{align*}
\zeta_{EZ, r}(s_1, \ldots, s_r | x_1, \ldots, x_r)&=\sum_{0<m_1<\cdots <m_r}\frac{1}{(m_1+x_1)^{s_1}\cdots (m_r+x_r)^{s_r}},\\
\zeta^{\star}_{EZ, r}(s_1, \ldots, s_r | x_1, \ldots, x_r)&=
\sum_{0< m_1\leq \cdots \leq m_r}\frac{1}{(m_1+x_1)^{s_1}\cdots (m_r+x_r)^{s_r}} 
\end{align*}
for any $x_i\geq 0$ and 
$s_i\in {\mathbb C}$ (see \cite{AkiyamaIshikawa}).   We further define
\begin{align*}
\zeta^{\star\star}_{EZ, r}(s_1, \ldots, s_r | x_1, \ldots, x_r)&=
\sum_{0\leq m_1\leq \cdots \leq m_r}\frac{1}{(m_1+x_1)^{s_1}\cdots (m_r+x_r)^{s_r}} 
\end{align*}
for any $x_i> 0$ and 
$s_i\in {\mathbb C}$. 
Hereafter we frequently omit the suffix $EZ,r$ on the left-hand sides, for brevity.

When all the $x_i$ are zeros, then they are original multiple zeta-function and zeta-star function of Euler-Zagier type, respectively:
\begin{align}
\zeta_{EZ, r}(s_1, \ldots, s_r | 0, \ldots, 0)&=\zeta_{EZ,r}(s_1,\ldots,s_r)=\sum_{0<m_1<\cdots<m_r}\frac{1}{m_1^{s_1}\cdots m_r^{s_r}},\label{EulerZagier}\\
\zeta^{\star}_{EZ, r}(s_1, \ldots, s_r | 0, \ldots, 0)&=\zeta^{\star}_{EZ,r}(s_1,\ldots,s_r)=\sum_{0<m_1\leq\cdots\leq m_r}\frac{1}{m_1^{s_1}\cdots m_r^{s_r}}.\notag
\end{align}
Also we see that $\zeta_{EZ,1}^{\star\star}(s | x)=\zeta(s,x)$.

Motivated by these extensions, we introduce the notion of Hurwitz type of Schur multiple 
zeta-functions, which is a combinatorial extension of all the above functions
with the structure of Schur functions.

We review the definition of the Schur multiple zeta-function, first.
Let $\lambda=(\lambda_1,\ldots,\lambda_m)$ be a partition of $n\in\mathbb{N}$.
We identify $\lambda$ with the corresponding Young diagram 
\[D(\lambda)=\{(i, j)\in {\mathbb Z}^2 ~|~ 1\leq i\leq r, 1\leq j\leq \lambda_i\},\] 
depicted as a collection of square boxes with $\lambda_i$ boxes in the $i$-th row.
Let $T(\lambda,X)$ the set of all Young tableaux of shape
$\lambda$ over a set $X$.    Let ${\rm SSYT}(\lambda)\subset T(\lambda,\mathbb{N})$ be 
the set of all
semi-standard Young tableaux of shape $\lambda$.
We write $M=(m_{ij})\in {\rm SSYT}(\lambda)$, where $m_{ij}$ denotes the number in the
$(i,j)$-box situated $i$-th row and $j$-th column of the tableaux.
The Schur multiple zeta-function associated with $\lambda$, introduced by
Nakasuji, Phuksuwan and Yamasaki in \cite{NPY18}, is defined by
\begin{align}\label{1-1}
\zeta_{\lambda}({\pmb s})=\sum_{M\in {\rm SSYT}(\lambda)}\prod_{(i,j)\in D(\lambda)}\frac{1}{m_{ij}^{s_{ij}}},
\end{align}
where ${\pmb s}=(s_{ij})\in T(\lambda,\mathbb{C})$.
This series converges absolutely if ${\pmb s}\in W_{\lambda}$ where 
\[
  W_\lambda =
\left\{{\pmb s}=(s_{ij})\in T(\lambda,\mathbb{C})\,\left|\,
\begin{array}{l}
 \text{$\Re(s_{ij})\ge 1$ for all $(i,j)\in \lambda \setminus C(\lambda)$ } \\[3pt]
 \text{$\Re(s_{ij})>1$ for all $(i,j)\in C(\lambda)$}
\end{array}
\right.
\right\}
\]
with $C(\lambda)$ being the set of all corners of $\lambda$.

Now it is natural to define the Hurwitz-type extension of Schur multiple zeta-functions as follows:
\begin{align}\label{1-1}
\zeta_{\lambda}({\pmb s}|{\pmb x})=\sum_{M\in {\rm SSYT}(\lambda)}\prod_{(i,j)\in D(\lambda)}\frac{1}{(m_{ij}+x_{ij})^{s_{ij}}}
\end{align}
for ${\pmb s}=(s_{ij})\in T(\lambda,\mathbb{C})$ and 
${\pmb x}=(x_{ij})\in T(\lambda,\mathbb{R}_{\geq 0})$.
This series also converges absolutely if ${\pmb s}\in W_{\lambda}$.
Hereafter sometimes we call $\zeta_{\lambda}({\pmb s}|{\pmb x})$ a Schur-Hurwitz
multiple zeta-function, for brevity.

Note that the Hurwitz type of Schur multiple zeta-function was first introduced 
by Haruki Yamamoto in his unpublished master thesis \cite{Yamamoto}.
In \cite{Yamamoto}, he called $\zeta_{\lambda}({\pmb s}|{\pmb x})$ a ``factorial
Schur multiple zeta-function'', because it may also be regarded as a zeta-analogue of 
factorial Schur functions.    Factorial Schur functions are the 4-th variant of
Schur functions in \cite{Mac}.    Some history
of factorial Schur functions (goes back to \cite{BL}) is described in \cite{BMN}.
The terminology ``factorial Schur multiple zeta-function'' is also used in
\cite{Minoguchi}, \cite{NakasujiTakeda2}. 

In this article, we will show that various known results that hold for ordinary Schur 
multiple zeta-functions can be extended to Hurwitz-type of Schur multiple zeta-functions.
\bigskip

Many parts of the present article were announced at the International Conference on
Lie Algebra and Number Theory, held at NIT Calicut, Kozhikode, India (June 2024).
The authors express their sincere gratitude to the organizers for the kind
invitation and hospitality.

\section{Determinant formulas for Schur multiple zeta-functions and expressions in terms of other multiple zeta-functions}\label{sec-2}

In this section we summarize several known facts on ordinary Schur multiple
zeta-functions $\zeta_{\lambda}({\pmb s})$, proved in \cite{NPY18}, \cite{NakasujiTakeda},
\cite{MN3} and \cite{MN4}.

In \cite{NPY18}, several determinant expressions of $\zeta_{\lambda}({\pmb s})$ were
established.    In particular, here we quote 
two important determinant formulas called  the ``Jacobi-Trudi formula" and 
the ``Giambelli formula" for
$\zeta_{\lambda}({\pmb s})$.

Let 
$$T^{\rm{diag}}(\lambda,\mathbb{C})=\{{\pmb s}=(s_{ij})\in T(\lambda,\mathbb{C})\,|\,\text{$s_{ij}=s_{lm}$ if $j-i=m-l$}\},$$
that is, all the variables on the same diagonal lines (from the upper left to the
lower right) are the same, and further define
$$ W_{\lambda}^{\rm{diag}}=W_{\lambda}\cap T^{\rm{diag}}(\lambda,\mathbb{C}).$$
When ${\pmb s}\in T^{\rm{diag}}(\lambda,\mathbb{C})$, we can introduce new variables 
$\pmb{z}=\{z_k\}_{k\in\mathbb{Z}}$ by the condition $s_{ij}=z_{j-i}$ (for all $i,j$),
and we may regard $\zeta_{\lambda}({\pmb s})$ as a function in variables 
$\{z_k\}_{k\in\mathbb{Z}}$.
We call the Schur multiple zeta-function associated with $\{z_k\}$ {\it content-parametrized Schur multiple zeta-function}, since $j-i$ is named ``content''.\\

 The Jacobi-Trudi formula for $\zeta_{\lambda}(\pmb{s})$ is as
follows.

\begin{thm}\label{JT}
{\rm (\cite[Theorem 1.1]{NPY18})}
For $\lambda=(\lambda_1, \ldots, \lambda_r)$ and its conjugate $\lambda'=(\lambda_1'. \ldots, , \lambda'_{r'})$,
let ${\pmb s}\in W_{\lambda}^{\rm{diag}}$. 
Then
\begin{align}\label{JT00-formula}
\zeta_{\lambda}(\pmb{s})=\det(\zeta^{\star}_{EZ,\lambda_i-i+j}(z_{-j+1},z_{-j+2},\ldots,
z_{-j+(\lambda_i-i+j)}))_{1\leq i,j\leq r},
\end{align}
and
\begin{align}\label{JT00-formula2}
\zeta_{\lambda}(\pmb{s})=\det(\zeta_{EZ,\lambda'_i-i+j}(z_{j-1},z_{j-2},\ldots,
z_{j-(\lambda'_i-i+j)}))_{1\leq i,j\leq r'}.
\end{align}
\end{thm}

\begin{remark}\label{convention}
We understand that $\zeta_{EZ,\lambda_i-i+j}(\cdots)=1\; ({\rm resp.} =0)$ if 
$\lambda_i-i+j=0 \;({\rm resp.} <0)$ in \eqref{JT00-formula}.
The same convention is also used in \eqref{JT00-formula2} and throughout this paper.
\end{remark}

\begin{remark}\label{method}
The proof is based on the method of the Lindstr\"{o}m-Gessel-Viennot lemma using the lattice path model, which was used to prove the Jacobi-Trudi formula for Schur functions. 
This lemma was proved by Gessel and Viennot in 1985 (\cite{GV}) based on Lindstr\"{o}m's paper in 1973 (\cite{L}).
In this lemma, we consider the total weight of the families of non-intersecting paths on a weighted directed acyclic graph with any given starting and ending points.
The key difference between the original proof for Schur function and the present proof for Schur multiple zeta-function is that the weights are appropriate for each.
\end{remark}

\begin{remark}\label{conti}
In this theorem we assume that ${\pmb s}\in W_{\lambda}^{\rm{diag}}$.    
However, since Schur multiple zeta-functions and Euler-Zagier multiple zeta-functions
can be continued meromorphically to the whole space, the formulas
\eqref{JT-formula} \eqref{JT-formula2} are valid on $T^{\rm{diag}}(\lambda,\mathbb{C})$
except for the singularity points of the relevant multiple zeta-functions.
This claim is applied to all other theorems in the present article.
\end{remark}

If we remove the condition on the variables on diagonal lines, these formulas in Theorem \ref{JT} do not hold. 
Nakasuji and Takeda \cite{NakasujiTakeda} obtained a partially unconditional formula 
by adding a process for summing variables.
We call this formula the ``extended Jacobi-Trudi formula":

 \begin{thm}\label{extendJT}
 {\rm (\cite[Theorem 3.4]{NakasujiTakeda})}
For $X\ge2$, $\lambda=(m,n,\{1\}^{X-2})$ and $(s_{ij})\in W_{\lambda,H}$ with any integers $m\ge n\ge1$, it holds that
\begin{align*}
&\sum_{diag}
\zeta_{\lambda}
\left(
\ytableausetup{boxsize=normal,aligntableaux=center}
\begin{ytableau}
  s_{11}&s_{12}&\cdots&\cdots&\cdots&s_{1m}\\
  s_{21}&s_{22}&\cdots&s_{2n}\\
s_{31}\\
\vdots\\
s_{X1}
\end{ytableau}\right)\\
&=\sum_{diag}\begin{vmatrix}
\zetas(\underline s_{\,11}^{\,1m})&\zetas(\underline s_{\,21}^{\,2n},\underline s_{\,1n}^{\,1m})&\zetas(s_{31},\underline s_{\,21}^{\,2n},\underline s_{\,1n}^{\,1m})&\cdots&\zetas(\underline s_{\,X1}^{\,31},\underline s_{\,21}^{\,2n},\underline s_{\,1n}^{\,1m})\\
\zetas(\underline s_{\,11}^{\,1(n-1)})&\zetas(\underline s_{\,21}^{\,2n})&\zetas(\underline s_{\,31},\underline s_{\,21}^{\,2n})&\cdots&\zetas(\underline s_{\,X1}^{\,31},\underline s_{\,21}^{\,2n})\\
0&1&\zetas(s_{31})&\cdots&\zetas(\underline s_{\,X1}^{\,31})\\
\vdots&\ddots&\ddots&\ddots&\vdots\\
0&\cdots&0&1&\zetas(s_{X1})
\end{vmatrix},
\end{align*}
where $W_{\lambda, H}$ and $\displaystyle\sum_{diag}$ are defined in Section $3$, and the notation $\underline{s}_{\,ia}^{\,ib}$ means $s_{ia}, s_{i(a+1)},\ldots, s_{ib}$ and $\underline{s}_{\, aj}^{\, bj}$ means $s_{aj}, s_{(a-1)j},\ldots, s_{bj}$.
\end{thm}

\begin{remark}
In \cite{NakasujiTakeda}, other cases such as $\lambda=(m,n,\{2\}^{X-2})$ are also
treated.
\end{remark}

Yamamoto obtained the Jacobi-Trudi formula for Schur-Hurwitz multiple 
zeta-functions and
Yuki Minoguchi generalized it to the extended Jacobi-Trudi formula in their master theses (\cite{Yamamoto}, \cite{Minoguchi}), respectively.
In Section \ref{sec-3}, we review their results. \\

Next, we show the Giambelli formula for $\zeta_{\lambda}({\pmb s})$ proved in \cite{NPY18}.

\begin{thm}\label{Giambelli0} 
{\rm (\cite[Theorem 4.5]{NPY18})}
Let  $\lambda$ be a partition whose
Frobenius notation is $\lambda=(p_1, \cdots , p_N | q_1, \cdots, q_N)$.
Assume ${\pmb s}\in W_{\lambda}^{\rm{diag}}$. 
Then we have
\begin{equation}\label{expij}
\zeta_{\lambda}({\pmb s}) = \det(\zeta_{i,j})_{1 \leq i,j \leq N},
\end{equation}
where $\zeta_{i,j}=\zeta_{(p_i+1, 1^{q_j})} ({\bf s}_{i,j}^F)$ with 
${\bf s}_{i,j}^F=
\ytableausetup{boxsize=normal}
  \begin{ytableau}
   z_0 & z_1 & z_2 &\cdots & z_{p_i}\\
   z_{-1}\\
   \vdots \\
    z_{-q_j}
  \end{ytableau}\in W_{(p_i+1, 1^{q_j})}.
$
Here, the Frobenius notation is defined in Section \ref{sec-4}.
 \end{thm}

\begin{remark}
In \cite{NPY18}, Theorem \ref{Giambelli0} is just reduced to the Giambelli-type formula
given in Nakagawa et al. \cite{NNSY01}.   
However, in another paper \cite{MN4}, we successfully proved it and extended to that for winged type in a different way
based on a paper of E{\v g}ecio{\v g}lu and Remmel \cite{EgeRem88}.    
In this method, we consider the product of the Schur (multiple zeta) functions of hook-type which are consisting of an arbitrary single row and an arbitrary single column of the Young tableau.
This method will also be used in this article.
 \end{remark}

We will discuss the Giambelli formula for Schur-Hurwitz multiple 
zeta-functions in Section \ref{sec-4}.\\

In \cite{MN3}, the authors proved the following expression for Schur multiple
zeta-functions of hook type in terms of multiple 
zeta-functions of Euler-Zagier type and their star-variants.
It is to be noticed that the same convention as in Remark \ref{convention} is applied
below.

\begin{thm}{\rm(\cite[Theorem 3.1]{MN3})}
\label{thm3}
For $\lambda=(p+1, 1^{q})$, we have
\begin{equation}\label{hook1}
\zeta_{\lambda}({\pmb s})=
\sum_{j=0}^{q}(-1)^j
\zeta^{\star}(z_{-j}, \ldots, z_{-1}, z_0, z_1, \ldots, z_p)
\zeta(z_{-j-1}, \ldots, z_{-q}),
\end{equation}
and
\begin{equation}\label{hook2}
\zeta_{\lambda}({\pmb s})=
\sum_{j=0}^{p}(-1)^j
\zeta(z_{j}, \ldots, z_{1}, z_0, z_{-1}, \ldots, z_{-q})
\zeta^{\star}(z_{j+1}, \ldots, z_{p}).
\end{equation}
\end{thm}

As a generalization of Theorem \ref{thm3}, the following theorem was also obtained.
By ${\frak S}_N$ we denote the $N$-th symmetric group.
\begin{thm}{\rm(\cite[Theorem 4.1]{MN3})}
\label{thm3b}
Let  $\lambda$ be a partition such that $\lambda=(p_1, \cdots , p_N | q_1, \cdots, q_N)$ 
in the Frobenius notation. 
Assume ${\pmb s}\in W_{\lambda}^{\rm{diag}}$. 
Then we have
\begin{align*}
\zeta_{\lambda}({\pmb s})=& \sum_{\sigma\in {\frak S}_N} {\rm{sgn}}(\sigma) \sum_{j_1=0}^{q_1} \cdots \sum_{j_N=0}^{q_N}
(-1)^{j_1+\cdots +j_N}\\
& \times \zeta^{\star}(z_{-j_1}, \ldots, z_0, \ldots, z_{p_{\sigma(1)}})
\zeta^{\star}(z_{-j_2},\ldots, z_0, \ldots, z_{p_{\sigma(2)}})\\
&\times\cdots\times
\zeta^{\star}(z_{-j_N}, \ldots, z_0, \ldots, z_{p_{\sigma(N)}})\\
& \times \zeta(z_{-j_1-1}, \ldots, z_{-q_1})
\zeta(z_{-j_2-1}, \ldots, z_{-q_2})
\ldots
\zeta(z_{-j_N-1}, \ldots, z_{-q_N}).
\end{align*}
\end{thm}

We will generalize the above theorem to the case of  Schur-Hurwitz multiple 
zeta-functions in Section \ref{sec-5}.

In \cite{MN3}, the authors also proved the following theorem:

\begin{thm}\label{mainthm3}
{\rm(\cite[Theorem 4.3]{MN3})}
Let  $\lambda$ be a partition whose Frobenius notation is 
$\lambda=(p_1, \cdots , p_N | q_1, \cdots, q_N)$.
Assume ${\pmb s}\in W_{\lambda}^{\rm{diag}}$. 
Then we have
\begin{align}\label{Dirichlet}
\zeta_{\lambda}({\pmb s})
=& \sum_{\substack{m_{11}, m_{22}, \ldots, m_{NN}\geq 1}}
(m_{11}\ldots m_{NN})^{-z_{0}}\\
&\times \sum_{\sigma\in {\frak S}_N} {\rm{sgn}}(\sigma)
\prod_{k=1}^N
\zeta_{EZ,p_k}^{\star\star}(z_1,\ldots,z_{p_k} | m_{\sigma(k)\sigma(k)},\ldots,
m_{\sigma(k)\sigma(k)})\notag\\
&\qquad\times\prod_{j=1}^N 
\zeta_{EZ,q_j}(z_{-1},\ldots,z_{-q_j} | m_{jj},\ldots,m_{jj}).\notag
\end{align}
\end{thm}

\begin{remark}
In \cite{MN3}, this theorem is stated in terms of zeta-functions of root
systems, instead of $ \zeta_{EZ,p_k}^{\star\star}$ and $\zeta_{EZ,q_j}$.
The zeta-functions of root
systems will be stated in Section \ref{sec-6}. 
\end{remark}

We will generalize Theorem \ref{mainthm3} to the case of Schur-Hurwitz multiple 
zeta-functions in Section \ref{sec-6}.

An advantage of introducing Schur-Hurwitz type of multiple zeta-functions is that
it is possible to consider various actions on shifting parameters.   In Section 
\ref{sec-7}, as an example of this viewpoint,
we discuss the differentiation with respect to shifting parameters,
which leads to (apparently) new formulas.


\section{The results by Yamamoto and by Minoguchi}\label{sec-3}
In this section, we review  
the Jacobi-Trudi formula for Schur-Hurwitz multiple 
zeta-functions obtained by Yamamoto
and the extended Jacobi-Trudi formula generalized by Minoguchi in their master theses (\cite{Yamamoto}, \cite{Minoguchi}), respectively.
Both of their results have not been published, so for the convenience of readers, we will include a sketch of the proofs of their results here. 
In the following subsections, we first recall several necessary notions.

\subsection{Rim decomposition of partitions}
 A skew Young diagram $\theta$ is a diagram obtained as a set difference of two Young diagrams of partitions $\lambda$ and $\mu$
satisfying $\mu\subset \lambda$, that is $\mu_i\le \lambda_i$ for all $i$.
In this case, we write $\theta=\lambda/\mu$.
It is called a {\it ribbon} if it is connected and contains no $2\times 2$ block of boxes.
 
Let $\lambda$ be a partition.
A sequence $\Theta=(\theta_1,\ldots,\theta_t)$ of ribbons is called a 
{\it rim decomposition of $\lambda$} if $D(\theta_k)\subset D(\lambda)$ for
$1\leq k\leq t$ and $\theta_1 \sqcup\cdots\sqcup \theta_k$ (the gluing of 
$\theta_1,\ldots,\theta_k$) is a partition $\lambda^{(k)}$ satisfying
$\lambda^{(t)}=\lambda$.
Here, $\theta_t$ is the maximal outer (that is, on the most bottom-right side) ribbon  
among them, which is called the {\it rim} of $\lambda$.
If we peel $\theta_t$ off the diagram $\lambda$, then the remaining is $\lambda^{(t-1)}$,
which has the rim $\theta_{t-1}$, and if we peel this rim then the remaining is
$\lambda^{(t-2)}$, and so on. 
 
 
\begin{example}
\label{ex:rim}
 The following $\Theta=(\theta_1,\theta_2,\theta_3,\theta_4)$ is a rim decomposition of $\lambda=(4,3,3,2)$;
\[
\ytableausetup{boxsize=normal,aligntableaux=center}
 \Theta
=\,
\begin{ytableau}
 1 & 2 & 3 & 4 \\
 1 & 2 & 3 \\
 1 & 3 & 3 \\
 3 & 3 
\end{ytableau}
,
\]
 which means that \!\!\!\!\!
\ytableausetup{boxsize=10pt,aligntableaux=center}
 $\theta_1=\ydiagram{1, 1, 1}$\,,
 $\theta_2=\ydiagram{0,1,1}$\,,
 $\theta_3=\ydiagram{2+1,2+1,1+2,2}$\, and
 $\theta_4=\ydiagram{1}$\,.
\end{example}

 Write $\lambda=(\lambda_1,\ldots,\lambda_r)$.
 We call a rim decomposition $\Theta=(\theta_1,\ldots,\theta_r)$ of $\lambda$ an {\it $H$-rim decomposition} 
 if each $\theta_i$ starts from $(i,1)$ for all $1\le i\le r$.
 Here, we permit $\theta_i=\emptyset$. 
 We denote by $\mathrm{Rim}^{\lambda}_H$ the set of all $H$-rim decompositions of $\lambda$.

\begin{example}
\label{ex:Hrim}
 The following $\Theta=(\theta_1,\theta_2,\theta_3,\theta_4)$ is an $H$-rim decomposition of $\lambda=(4,3,3,2)$; 
\[
\ytableausetup{boxsize=normal,aligntableaux=center}
 \Theta
=\,
\begin{ytableau}
 1 & 1 & 3 & 3 \\
 3 & 3 & 3 \\
 3 & 4 & 4 \\
 4 & 4 
\end{ytableau}
,
\]
 which means that \!\!\!\!\!
\ytableausetup{boxsize=10pt,aligntableaux=center} 
 $\theta_1=\ydiagram{2}$\,,
 $\theta_2=\emptyset$\,, 
 $\theta_3=\ydiagram{2+2,3,1}$\, and 
 $\theta_4=\ydiagram{1+2,2}$\,.
 Note that the rim decomposition appeared in Example~\ref{ex:rim} is not an $H$-rim decomposition.
\end{example}
Also, a rim decomposition $\Theta=(\theta_1,\ldots,\theta_{r'})$ of $\lambda$ is called
 an {\it $E$-rim decomposition} if each $\theta_i$ starts from $(1,i)$ for all 
 $1 \le i \le r'$. 
 Here, we again permit $\theta_i=\emptyset$.
 We denote by $\mathrm{Rim}^{\lambda}_E$ the set of all $E$-rim decompositions of $\lambda$.
 Note that the rim decomposition appeared in Example~\ref{ex:rim} is an $E$-rim decomposition.

\subsection{Patterns on the $\mathbb{Z}^2$ lattice}

 Fix $N\in\mathbb{N}$.
 For a partition $\lambda=(\lambda_1,\ldots,\lambda_r)$,
 let $a_i$ and $b_i$ be lattice points in $\mathbb{Z}^2$
 respectively given by $a_i=(r+1-i,1)$ and $b_i=(r+1-i+\lambda_i,N)$ for $1\le i\le r$.
 Put $A=(a_1,\ldots,a_r)$ and $B=(b_1,\ldots,b_r)$.
 An {\it $H$-pattern} corresponding to $\lambda$ is a tuple $L=(l_1,\ldots,l_r)$ of directed paths on $\mathbb{Z}^2$,
 whose directions are allowed only to go one to the right or one up,
 such that $l_i$ starts from $a_i$ and ends to $b_{\sigma(i)}$ for some permutation $\sigma\in \mathfrak S_r$. 
 We call such $\sigma\in\mathfrak  S_r$ the {\it type} of $L$ and denote it by $\sigma=\mathrm{type}(L)$.
 Note that the type of an $H$-pattern does not depend on $N$.
 Let $\mathcal{H}^{N}_{\lambda}$ be the set of all $H$-patterns corresponding to $\lambda$.
 

We can also consider similar argument for $E$-rim.
 Let $c_i$ and $d_i$ be lattice points in $\mathbb{Z}^2$
 respectively given by $c_i=(r'+1-i,1)$ and $d_i=(r'+1-i+\lambda'_i,N+1)$ for $1\le i\le r'$.
 Put $C=(c_1,\ldots,c_{r'})$ and $D=(d_1,\ldots,d_{r'})$.
 An {\it $E$-pattern} corresponding to $\lambda$ is a tuple $L=(l_1,\ldots,l_{r'})$ of directed paths on $\mathbb{Z}^2$,
 whose directions are allowed only to go one to the northeast or one up,
 such that $l_i$ starts from $c_i$ and ends to $d_{\sigma(i)}$ for some $\sigma\in \mathfrak S_{r'}$.
 We also call such $\sigma\in\mathfrak S_{r'}$ the {\it type} of $L$ and denote it by $\sigma=\mathrm{type}(L)$.
Let $\mathcal{E}^{N}_{\lambda}$ be the set of all $E$-patterns corresponding to $\lambda$.

\subsection{Weight of patterns}

 Fix ${\pmb s}=(s_{ij})\in W_\lambda$ and ${\pmb x}=(x_{ij})\in T(\lambda, {\mathbb R}_{\geq 0})$.
 We next assign a weight to $L=(l_1,\ldots,l_r) \in \mathcal{H}^{N}_{\lambda}$ via the $H$-rim decomposition of $\lambda$ as follows.
 Take $\Theta=(\theta_1,\ldots,\theta_r)\in\mathrm{Rim}^{\lambda}_{H}$ such that $\tau_H(\Theta)=\mathrm{type}(L)$ which is the biijection map $\tau_H:\mathrm{Rim}^{\lambda}_H\to {\frak S}^{\lambda}_H$,
 where ${\frak S}^{\lambda}_H=\{\mathrm{type}(L)\in {\frak S}_r | L\in \mathcal{H}^{N}_{\lambda}\}$.
 It is known that the number of horizontal edges appearing in the path $l_i$ is equal to the number of boxes of $\theta_i$ (see \cite[Lemma 3.4]{NPY18}). 
 Then, when the $k$-th horizontal edge of $l_i$ is on the $j$-th row,
 we weight it with $\displaystyle{\frac{1}{(j+x_{pq})^{s_{pq}}}}$ 
 where $(p,q)\in D(\lambda)$ is the $k$-th component of $\theta_i$.
 Now, the weight $w^N_{{\pmb s}}(l_i)$ of the path $l_i$ is defined to be the product of weights of all horizontal edges along $l_i$.
 Here, we understand that $w^{N}_{{\pmb s}}(l_i)=1$ if $\theta_i=\emptyset$.
 Moreover, we define the weight $w^N_{{\pmb s}}(L)$ of $L\in \mathcal{H}^{N}_{\lambda}$ by 
\begin{equation}\label{defweight}
 w^N_{{\pmb s}}(L)=\prod^{r}_{i=1}w^{N}_{{\pmb s}}(l_i).
\end{equation}
See Example 3.3 below.  

Similarly, we also define a weight on $L=(l_1, \ldots, l_{r'})\in \mathcal{E}_{\lambda}^{N}$ via the $E$-rim decomposition of $\lambda$. For $\Theta=(\theta_1, \ldots, \theta_{r'})\in \rm{Rim}_E^{\lambda}$ such that
$\tau_E(\Theta)={\rm type} (L)$ with the bijection map 
$$\tau_E:\mathrm{Rim}^{\lambda}_E\to {\frak S}^{\lambda}_E=\{\mathrm{type}(L)\in {\frak S}_{r'} | L\in \mathcal{E}^{N}_{\lambda}\},$$
we weight the $k$-th northeast edge of $l_i$ which lies from the $j$-th row to $(j+1)$-th row with $\dfrac{1}{(j+x_{pq})^{s_{pq}}}$,  where $(p, q)\in D(\lambda)$ is the $k$-th component of $\theta_i$.
We define the weight of the path $l_i$, $w_{\pmb s}^N(l_i)$, by the product of weights of all northeast edges along $l_i$ and the weight $w_{\pmb s}^N(L)$ of $L\in \mathcal{E}_{\lambda}^N$ by \eqref{defweight}.

\begin{example}
 Let $\lambda=(4,3,3,2)$. For $N=4$, 
we consider the following $L=(l_1,l_2,l_3,l_4)\in \mathcal{H}^{N}_{\lambda}$;
\begin{figure}[ht]
\begin{center}
 \begin{tikzpicture} 
  \node at (0.5,1) {$1$};
  \node at (0.5,2) {$2$};
  \node at (0.5,3) {$3$};
  \node at (0.5,4) {$4$};
     \node at (1,0) {$a_4$};  
     \node at (2,0) {$a_3$};
     \node at (3,0) {$a_2$};
     \node at (4,0) {$a_1$};
     \node at (1,0.5) {$(1,1)$};  
     \node at (2,0.5) {$(2,1)$};
     \node at (3,0.5) {$(3,1)$};
     \node at (4,0.5) {$(4,1)$};
       \node at (3,4.5) {$(3,4)$};
       \node at (5,4.5) {$(5,4)$};
       \node at (6,4.5) {$(6,4)$}; 
       \node at (8,4.5) {$(8,4)$};
       \node at (3,5) {$b_4$};
       \node at (5,5) {$b_3$};
       \node at (6,5) {$b_2$}; 
       \node at (8,5) {$b_1$}; 
   \node at (1,1) {$\bullet$};
   \node at (1,2) {$\bullet$};
   \node at (1,3) {$\bullet$};
   \node at (1,4) {$\bullet$};    
   \node at (2,1) {$\bullet$};
   \node at (2,2) {$\bullet$};
   \node at (2,3) {$\bullet$};  
   \node at (2,4) {$\bullet$};  
   \node at (3,1) {$\bullet$};
   \node at (3,2) {$\bullet$};
   \node at (3,3) {$\bullet$};  
  \node at (3,4) {$\bullet$};  
   \node at (4,1) {$\bullet$};
   \node at (4,2) {$\bullet$};
   \node at (4,3) {$\bullet$};  
   \node at (4,4) {$\bullet$};  
   \node at (5,1) {$\bullet$};
   \node at (5,2) {$\bullet$};
   \node at (5,3) {$\bullet$};  
   \node at (5,4) {$\bullet$};  
   \node at (6,1) {$\bullet$};
   \node at (6,2) {$\bullet$};
   \node at (6,3) {$\bullet$}; 
   \node at (6,4) {$\bullet$};  
   \node at (7,1) {$\bullet$};
   \node at (7,2) {$\bullet$};
   \node at (7,3) {$\bullet$}; 
   \node at (7,4) {$\bullet$};  
   \node at (8,1) {$\bullet$};
   \node at (8,2) {$\bullet$};
   \node at (8,3) {$\bullet$};  
   \node at (8,4) {$\bullet$};  
 \draw (4,1) -- (5,1) -- (6,1) -- (6,2) -- (6,3) -- (6,4);
 \draw[dotted] (3,1) -- (3,2) -- (3,3) -- (3,4);
 \draw[dashed] (2,1) -- (2,2) -- (3,2) -- (4,2) -- (5,2)-- (6,2)-- (7,2)-- (7,3)-- (8,3)-- (8,4) ;      
 \draw[loosely dashdotted] (1,1) -- (1,2) -- (1,3) -- (2,3) -- (3,3) -- (4,3) -- (5,3) -- (5,4);
 \end{tikzpicture}
\end{center}
\ \\[-10pt]
\caption{$L=(l_1,l_2,l_3,l_4)\in \mathcal{H}^{4}_{(4,3,3,2)}$}
\end{figure}
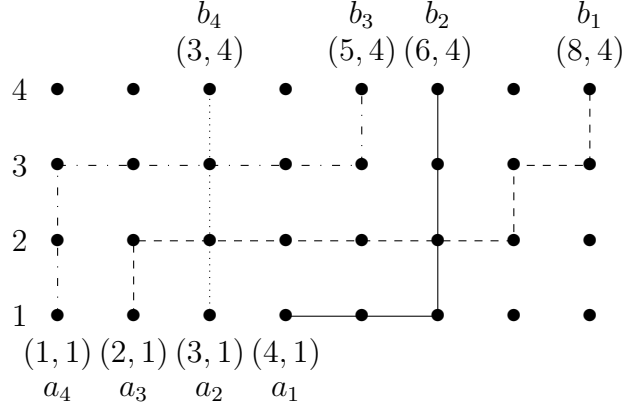

 Since $\mathrm{type}(L)=(1243)$,
 the corresponding $H$-rim decomposition of $\lambda$ is nothing but the one
 appeared in Example~\ref{ex:Hrim}. 
 For
$$
{\pmb s}=
\ytableausetup{boxsize=normal}
  \begin{ytableau}
   s_{11} & s_{12} & s_{13} & s_{14}\\
   s_{21} & s_{22} & s_{23} \\
   s_{31}  & s_{32} & s_{33} \\
   s_{41}  & s_{42}
  \end{ytableau}\in W_{\lambda}
\;\;{\rm and}\;\; 
{\pmb x}=
\ytableausetup{boxsize=normal}
  \begin{ytableau}
   x_{11} & x_{12} & x_{13} & x_{14}\\
   x_{21} & x_{22} & x_{23} \\
   x_{31}  & x_{32} & x_{33} \\
   x_{41}  & x_{42}
  \end{ytableau}\in T({\lambda}, {\mathbb R}_{\geq 0}),
$$ 
we have
{\small
\begin{align*}
w^{N}_{{\pmb s}}(l_1)&=\dfrac{1}{(1+x_{11})^{s_{11}}(1+x_{12})^{s_{12}}}, \quad
w^{N}_{{\pmb s}}(l_2)=1,\\
w^{N}_{{\pmb s}}(l_3)&=\dfrac{1}{(2+x_{31})^{s_{31}}(2+x_{21})^{s_{21}}(2+x_{22})^{s_{22}}(2+x_{23})^{s_{23}}(2+x_{13})^{s_{13}}(3+x_{14})^{s_{14}}},\\
w^{N}_{{\pmb s}}(l_4)&=\dfrac{1}{(3+x_{41})^{s_{41}}(3+x_{42})^{s_{42}}(3+x_{32})^{s_{32}}(3+x_{33})^{s_{33}}}.
\end{align*}
}%
\end{example}
${}$\\

\subsection{A result by Yamamoto}

In this subsection we briefly outline the argument by Yamamoto \cite{Yamamoto}
 which is almost the same as 
in \cite{NPY18}, using
the method of the Lindstr\"{o}m-Gessel-Viennot lemma (see Remark \ref{method}),
to yield the following Jacobi-Trudi formula for Schur-Hurwitz multiple zeta-functions.

Similarly to $T^{\rm diag}(\lambda,\mathbb{C})$, let
$$
T^{\rm diag}(\lambda,\mathbb{R}_{\geq 0})
=\{{\pmb x}=(x_{ij})\in T(\lambda,\mathbb{R}_{\geq 0})\mid 
x_{ij}=x_{lm} \;{\rm if}\; j-i=m-l \}.
$$
When ${\pmb x}\in T^{\rm diag}(\lambda,\mathbb{R}_{\geq 0})$, we can introduce new
parameters ${\pmb y}=\{y_k\}_{k\in \mathbb{Z}}$ by $x_{ij}=y_{j-i}$ (for all $i,j$).

\begin{thm}\label{JTYamamoto}
{\rm (\cite[Theorem 4.2]{Yamamoto})}
For $\lambda=(\lambda_1, \ldots, \lambda_r)$ and its conjugate $\lambda'=(\lambda_1'. \ldots, , \lambda'_{r'})$,
let ${\pmb s}=(s_{ij})\in W_{\lambda}^{\rm{diag}}$ with $s_{ij}=z_{j-i}$ and ${\pmb x}=(x_{ij})\in T^{\rm{diag}}({\lambda}, {\mathbb R}_{\geq 0})$ with $x_{ij}=y_{j-i}$. 
Then
\begin{align}\label{JT-formula}
\zeta_{\lambda}(\pmb{s}| \pmb{x})
=\det(\zeta_{EZ,\lambda_i-i+j}^{\star}({\bf z}_H|{\bf y}_H))_{1\leq i,j\leq r},
\end{align}
where 
\begin{align*}
({\bf z}_H)&=(z_{-j+1},z_{-j+2},\ldots,
z_{-j+(\lambda_i-i+j)})\\
 ({\bf y}_H)&=(y_{-j+1},y_{-j+2},\ldots, y_{-j+(\lambda_i-i+j)}),
 \end{align*}
and
\begin{align}\label{JT-formula2}
\zeta_{\lambda}(\pmb{s}| \pmb{x})=\det(\zeta_{EZ,\lambda'_i-i+j}({\bf z}_E|{\bf y}_E))_{1\leq i,j\leq r'},
\end{align}
where 
\begin{align*}
({\bf z}_E)&=(z_{j-1},z_{j-2},\ldots,
z_{j-(\lambda'_i-i+j)})\\
 ({\bf y}_E)&=(y_{j-1},y_{j-2},\ldots,
y_{j-(\lambda'_i-i+j)}).
 \end{align*}
\end{thm}

The point of the method of the Lindstr\"{o}m-Gessel-Viennot lemma is that the function in question can be described in terms of a lattice path model by the bijection between the semi-standard Young tableaux of shape $\lambda$ and non-intersecting weighted lattice paths. Indeed,
let $\mathcal{H}_{\lambda, 0}^{N}\subset \mathcal{H}_{\lambda}^{N}$ be the set of all $L=(l_1, \ldots, l_r)\in \mathcal{H}_{\lambda}^N$
such that any distinct pair of $l_i$ and $l_j$ has no intersection. We easily find $\mbox{type}(L)=\1$ for all $L\in \mathcal{H}_{\lambda, 0}^N$ where $\1$ is the identity element of $\mathfrak S_r$, and the above bijection gives
$$\zeta_{\lambda}({\pmb s} | {\pmb x})=\sum_{L\in \mathcal{H}_{\lambda, 0}^N}w_{\pmb s}^N(L).
$$

To derive the desired determinant formula, we need to consider the sum
\begin{equation}\label{pathexpression}
\sum_{L\in \mathcal{H}_{\lambda}^N}{\rm{sgn}}({\rm{type}}(L))w_{\pmb s}^N(L),
\end{equation}
where ${\rm sgn}({\rm{type}}(L))$ is the signature of ${\rm{type}}(L)$.
Therefore, the necessary remaining task is to show that the sum over paths with intersections disappears.
As we can see in Figures $2$ and $3$ in Example \ref{fig2fig3},
since the diagonal components in ${\pmb s}\in W_{\lambda}$ appear in the same column in a lattice path model,
we find that there are pairs of paths with the same weight 
that differ from each other only by a sign, so that
we have the following lemma.

\begin{lem}\label{cancellationforL}
If ${\pmb s}\in W_{\lambda}^{\rm{diag}}$,
for any $L\in \mathcal{H}_{\lambda}^N\backslash \mathcal{H}_{\lambda, 0}^N$, 
there exists another element $\overline{L}\in \mathcal{H}_{\lambda}^N\backslash \mathcal{H}_{\lambda, 0}^N$ such that
\begin{align}\label{daggerzero}
{\rm{sgn}}({\rm{type}}(L))w_{\pmb s}^N(L)+{\rm{sgn}}({\rm{type}}(\overline{L}))w_{\pmb s}^N(\overline{L})=0.
\end{align}
Therefore, it holds that
$$
\sum_{L\in \mathcal{H}_{\lambda}^N\backslash \mathcal{H}_{\lambda, 0}^N}{\rm{sgn}}({\rm{type}}(L))w_{\pmb s}^N(L)=0.
$$
\end{lem}

By Lemma \ref{cancellationforL}, we see that $\zeta_{\lambda}({\pmb s} | {\pmb x})$ 
is equal to \eqref{pathexpression}.    
Taking the limit $N\to \infty$, we see that the resulting expression can be written in terms of multiple zeta-star functions, hence we obtain \eqref{JT-formula}.
The similar discussion holds for $E$-type.

\begin{example}\label{fig2fig3}
In case of $\lambda=(3,2)$, for ${\pmb s}=\ytableausetup{boxsize=normal}
  \begin{ytableau}
  a & b &c \\
  d & e
 \end{ytableau}
$, Figures \ref{cap2} and \ref{cap3} show the corresponding path models $L$ and $\overline{L}$ for different $H$-rim decomposition, respectively.
\begin{figure}[ht]
\begin{center}
 \begin{tikzpicture} 
  \node at (0.5,1) {$1$};
  \node at (0.5,2) {$2$};
  \node at (0.5,3) {$3$};
  \node at (0.5,4) {$4$};
  \node at (1,0.5) {$l_2$};  
   \node at (2,0.5) {$l_1$};
  \node at (2.5,2.2) {{\small$a$}};
  \node at (3.5,3.2) {{\small$b$}};
    \node at (4.5,3.2) {{\small$c$}};
      \node at (1.5,3.2) {{\small$d$}};
        \node at (2.5,3.2) {{\small$e$}};
    \node at (1,1) {$\bullet$};
   \node at (1,2) {$\bullet$};
   \node at (1,3) {$\bullet$};
   \node at (1,4) {$\bullet$};    
   \node at (2,1) {$\bullet$};
   \node at (2,2) {$\bullet$};
   \node at (2,3) {$\bullet$};  
   \node at (2,4) {$\bullet$};  
   \node at (3,1) {$\bullet$};
   \node at (3,2) {$\bullet$};
   \node at (3,3) {$\bullet$};  
  \node at (3,4) {$\bullet$};  
   \node at (4,1) {$\bullet$};
   \node at (4,2) {$\bullet$};
   \node at (4,3) {$\bullet$};  
   \node at (4,4) {$\bullet$};  
   \node at (5,1) {$\bullet$};
   \node at (5,2) {$\bullet$};
   \node at (5,3) {$\bullet$};  
   \node at (5,4) {$\bullet$};  
 \draw(2,1) -- (2,2) -- (3,2) -- (3,3) -- (3,4);
 \draw[dashed] (1,1) -- (1,2) -- (1,3) -- (2,3) -- (3,3) -- (4,3) -- (5,3) --  (5,4);
 \end{tikzpicture}
\end{center}
\ \\[-10pt]
\caption{$L=(l_1,l_2)$}\label{cap2}
\end{figure}
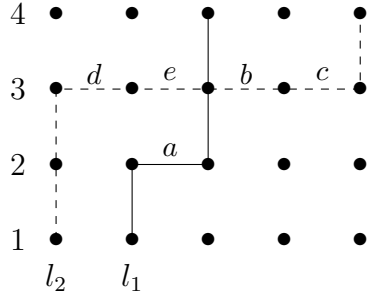

\begin{figure}[ht]
\begin{center}
 \begin{tikzpicture} 
  \node at (0.5,1) {$1$};
  \node at (0.5,2) {$2$};
  \node at (0.5,3) {$3$};
  \node at (0.5,4) {$4$};
  \node at (1,0.5) {$\overline{l_2}$};  
   \node at (2,0.5) {$\overline{l_1}$};
     \node at (2.5,2.2) {{\small$e$}};
  \node at (3.5,3.2) {{\small$b$}};
    \node at (4.5,3.2) {{\small$c$}};
      \node at (1.5,3.2) {{\small$d$}};
        \node at (2.5,3.2) {{\small$a$}};
   \node at (1,1) {$\bullet$};
   \node at (1,2) {$\bullet$};
   \node at (1,3) {$\bullet$};
   \node at (1,4) {$\bullet$};    
   \node at (2,1) {$\bullet$};
   \node at (2,2) {$\bullet$};
   \node at (2,3) {$\bullet$};  
   \node at (2,4) {$\bullet$};  
   \node at (3,1) {$\bullet$};
   \node at (3,2) {$\bullet$};
   \node at (3,3) {$\bullet$};  
  \node at (3,4) {$\bullet$};  
   \node at (4,1) {$\bullet$};
   \node at (4,2) {$\bullet$};
   \node at (4,3) {$\bullet$};  
   \node at (4,4) {$\bullet$};  
   \node at (5,1) {$\bullet$};
   \node at (5,2) {$\bullet$};
   \node at (5,3) {$\bullet$};  
   \node at (5,4) {$\bullet$};  
 \draw(2,1) -- (2,2) -- (3,2) -- (3,3) -- (4,3) -- (5,3) --  (5,4);
 \draw[dashed] (1,1) -- (1,2) -- (1,3) -- (2,3) -- (3,3) -- (3,4);
 \end{tikzpicture}
\end{center}
\ \\[-10pt]
\caption{$\overline{L}=(\overline{l_1},\overline{l_2})$}\label{cap3}
\end{figure}
\noindent
Since ${\rm type}(L)=(12)$ and ${\rm type}(\overline{L})={\1}$ so that ${\rm sgn}({\rm type}(L))=-{\rm sgn}({\rm type}(\overline{L}))$, if ${\pmb s}\in W_{\lambda}^{\rm{diag}}$ and 
${\pmb x}\in T^{\rm{diag}}(\lambda, {\mathbb R}_{\geq 0})$, that is $a=e$ and $x_{11}=x_{22}$, then we get
\begin{equation}\label{disappear}
{\rm{sgn}}({\rm{type}}(L))w_{\pmb s}^N(L)+{\rm{sgn}}({\rm{type}}(\overline{L}))w_{\pmb s}^N(\overline{L})=0.
\end{equation}
\end{example}

\noindent

\subsection{A result by Minoguchi}

Let $I(j)=\{(k, \ell)\in D(\lambda) | \ell-k=j\}$,
$S_j$ be the set of all permutations on $I(j)$,
and $Z(\lambda)=\{j\in {\mathbb Z} | I(j)\neq\emptyset\}$.
For any function $f$ defined on $T(\lambda,X)$, we define
\begin{align}\label{diag_def}
\sum_{diag}f((t_{k,\ell})_{(k,\ell)\in D(\lambda)})
\quad:=\sum_{\substack{\sigma_j\in {\frak S}_{|I(j)|} \\ (j\in Z(\lambda))}}
f\left((t_{\sigma_j(k,\ell)})_{(k,\ell)\in I(j), j\in Z(\lambda)}
\right)
\end{align}
where $(t_{k,\ell})_{(k,\ell)\in D(\lambda)}\in T(\lambda,X)$ and
$|I(j)|$ is the cardinality of $I(j)$.\\
\begin{example}
For $\lambda=(4, 3)$,
\begin{align*}
 I(-1)&=\{(k, \ell)\in D(\lambda) | \ell-k=-1\}=\{(2,1)\},\\
 I(0)&=\{(k, \ell)\in D(\lambda) | \ell-k=0\}=\{(1,1), (2,2)\},\\
 I(1)&=\{(k, \ell)\in D(\lambda) | \ell-k=1\}=\{(1,2), (2,3)\},\\
 I(2)&=\{(k, \ell)\in D(\lambda) | \ell-k=2\}=\{(1,3)\},\\
 I(3)&=\{(k, \ell)\in D(\lambda) | \ell-k=3\}=\{(1,4)\}.
 \end{align*}
This leads to 
\begin{align*}
&S_{-1}=S_{2}=S_{3}\cong {\frak S}_1=\{\1\},\\
&S_0\cong {\frak S}_2=\{\1, \sigma_1\},\; S_1\cong {\frak S}_2=\{\1, \sigma_1\},
\end{align*}
where $\sigma_1$ implies the substitution of the first and second components of $I(j)$ for any $j$. 
Therefore, 
\begin{align*}
\sum_{diag} & f\left(
\ytableausetup{boxsize=normal}
  \begin{ytableau}
   a & b & c & d\\
   e & f & g
  \end{ytableau}\right)=&\\
 &\quad f\left(\ytableausetup{boxsize=normal}
  \begin{ytableau}
   a & b & c & d\\
   e & f & g
  \end{ytableau}\right) &(\1, \1, \1, \1, \1)\in S_{-1}\times S_0\times S_1\times S_2\times S_3\\
  &+
  f\left(\ytableausetup{boxsize=normal}
  \begin{ytableau}
   f & b & c & d\\
   e & a & g
  \end{ytableau}\right) &(\1, \sigma_1, \1, \1, \1)\in S_{-1}\times S_0\times S_1\times S_2\times S_3\\
  &+
 f \left(\ytableausetup{boxsize=normal}
  \begin{ytableau}
   a & g & c & d\\
   e & f & b
  \end{ytableau}\right) &(\1, \1, \sigma_1, \1, \1)\in S_{-1}\times S_0\times S_1\times S_2\times S_3\\
 &+
 f \left(\ytableausetup{boxsize=normal}
  \begin{ytableau}
   f & g & c & d\\
   e & a & b
  \end{ytableau}\right) &(\1, \sigma_1, \sigma_1, \1, \1)\in S_{-1}\times S_0\times S_1\times S_2\times S_3.
  \end{align*}
\end{example}

Now we present the result of Minoguchi \cite{Minoguchi} on the extended Jacobi-Trudi
formula for Schur-Hurwitz multiple zeta-functions.

For $H(\lambda):=\{(i, j)\in D(\lambda) | j-i\in \{\lambda_i-i | 1\leq i \leq r\}\}$, we define
$$W_{\lambda, H}=\left\{{\pmb s}=(s_{ij}\in T({\lambda}, \mathbb{C}) \left| 
\begin{array}{ll}
\Re(s_{ij})\geq 1& \mbox{ for all } (i,j)\in D(\lambda)\backslash H(\lambda)\\
\Re(s_{ij})>1 & \mbox{ for all } (i,j)\in H(\lambda)
\end{array}\right.\right\}.
$$
Then $\sum_{diag}\zeta_{\lambda}({\pmb s} | {\pmb x})$ converges absolutely if ${\pmb s}\in W_{\lambda, H}$.

As we see in Example \ref{fig2fig3}, if ${\pmb s}\not\in  W_{\lambda}^{diag}$ or ${\pmb x}\not\in  T^{diag}({\lambda}, {\mathbb R}_{\geq 0})$, then  \eqref{disappear} does not hold. However, by summing over the permutations, we again encounter the disappearing phenomenon.
For example,
\begin{align*}
\sum_{diag}\zeta_{\lambda}({\pmb s} | {\pmb x})&=
\zeta_{\lambda}\left(\ytableausetup{boxsize=normal}
  \begin{ytableau}
  a & b &c \\
  d & e
 \end{ytableau}
 \left| 
 \ytableausetup{boxsize=normal}
  \begin{ytableau}
  x_{11} & x_{12} & x_{13} \\
  x_{21} & x_{22}
 \end{ytableau}
\right.\right)
+
\zeta_{\lambda}\left(\ytableausetup{boxsize=normal}
  \begin{ytableau}
  e & b &c \\
  d & a
 \end{ytableau}
 \left| 
 \ytableausetup{boxsize=normal}
  \begin{ytableau}
  x_{22} & x_{12} & x_{13} \\
  x_{21} & x_{11}
 \end{ytableau}
\right.\right)\\
&=0
\end{align*}
in the case of Example \ref{fig2fig3}.
This type of result can be derived generally
for the case $\lambda=(m, n, \{1\}^{X-2})$ ($X\geq 2$) by 
the induction argument on $j$, quite similar to that in \cite{NakasujiTakeda},
 and hence the following theorem holds:

 \begin{thm}\label{extendJTminoguchi}
 {\rm (\cite[Theorem 9.2]{Minoguchi})}
For $X\ge2$, $\lambda=(m,n,\{1\}^{X-2})$ and $(s_{ij})\in W_{\lambda,H}$ with any integers $m\ge n\ge1$, it holds that
\begin{align*}
&\sum_{diag}\zeta_{\lambda}\left(\begin{ytableau}
  s_{11}&s_{12}&\cdots&\cdots&\cdots&s_{1m}\\
  s_{21}&s_{22}&\cdots&s_{2n}\\
s_{31}\\
\vdots\\
s_{X1}
\end{ytableau}\left| 
\begin{ytableau}
  x_{11}&x_{12}&\cdots&\cdots&\cdots&x_{1m}\\
  x_{21}&x_{22}&\cdots&x_{2n}\\
x_{31}\\
\vdots\\
x_{X1}
\end{ytableau}
\right.\right)\\
&=\sum_{diag}\begin{vmatrix}
\zetas(\underline s_{\,11}^{\,1m} | \underline x_{\,11}^{\,1m})&
*&*&\cdots&*\\
\zetas(\underline s_{\,11}^{\,1(n-1)}|\underline x_{\,11}^{\,1(n-1)})&\zetas(\underline s_{\,21}^{\,2n}|\underline x_{\,21}^{\,2n})&*&\cdots&*\\
0&1&\zetas(s_{31}|x_{31})&\cdots&*\\
\vdots&\ddots&\ddots&\ddots&\vdots\\
0&\cdots&0&1&\zetas(s_{X1}|x_{X1})
\end{vmatrix},
\end{align*}
where $(1,2)$-component of the matrix on the right-hand side is 
$$\zetas(\underline s_{\,21}^{\,2n}, s_{\,1n}^{\,1m}|\underline x_{\,21}^{\,2n}, x_{\,1n}^{\,1m}),$$
$(1,j)$-component  ($ j\geq 3$) is 
$$
\zetas(\underline s_{\,j1}^{\,31},\underline s_{\,21}^{\,2n}
,\underline s_{\,1n}^{\,1m}|\underline x_{\,j1}^{\,31},\underline x_{\,21}^{\,2n}
,\underline x_{\,1n}^{\,1m}),
$$
$(2,j)$-component ($ j\geq 3$) is
$\zetas(\underline s_{\,j1}^{\,31}, \underline s_{\,21}^{\,2n}|\underline x_{\,j1}^{\,31}, 
\underline x_{\,21}^{\,2n})$ and $(i,j)$-component $(i\geq 3)$ is 
$\zetas(\underline s_{\,j1}^{\,i1}|\underline x_{\,j1}^{\,i1})$.
\end{thm}


\section{Giambelli formula for Schur-Hurwitz multiple zeta-functions}\label{sec-4}

Now we are going into the main body of the present paper.
Several new results will be proved in the following sections.
First, in this section, we give Giambelli-type formulas for Schur-Hurwitz multiple
zeta-functions, by the method quite similar to that developed in \cite{MN4}.

\subsection{Frobenius notation}
For a partition $\lambda=(\lambda_1,\lambda_2,\ldots)$, let $\lambda'=(\lambda_1',\lambda_2',\ldots)$ be its
conjugate, and
we define two sequences of indices
$p_1, \ldots, p_N$ and $q_1, \ldots, q_N$ by $p_i=\lambda_i-i$ and $q_i=\lambda'_i-i$ for $1\leq i \leq N$
where $N$ is the number of the main diagonal entries of the Young diagram of $\lambda$.
We sometimes write $\lambda=(p_1, \ldots, p_N | q_1, \ldots, q_N)$, 
the Frobenius notation of $\lambda$ (see \cite[Section 1.1]{Mac95}).
Let $\alpha(\lambda)\subset \lambda$ be the set of all cells 
which are to the right of the northeast boundary of the main diagonal in the Young diagram of shape $\lambda$ and
$\beta(\lambda)\subset \lambda$ be the set of all cells 
which are to the left of the southwest boundary of the main diagonal in the Young diagram of shape $\lambda$.
We denote them by $\alpha(\lambda)=(p_1, \ldots, p_N)$, $\beta(\lambda)=(q_1, \ldots, q_N)$. 
Moreover, 
let $D^M(\lambda)$ be the set of all cells 
which are on the main diagonal in the Young diagram of shape $\lambda$.
\begin{example}
For $\lambda=(3,3,2,1)$, the Young diagram of shape $\lambda$ is depicted as
$$
  \ydiagram[]{3,3,2,1}\; .
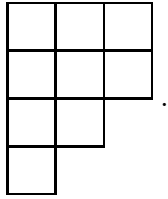
$$
$\alpha(\lambda)=(2,1)$, depicted as 
$
\ydiagram{2,1+1}
$\; ,
and $\beta(\lambda)=(3,1)$, 
which is depicted as
$
\ydiagram{1,2,1}
$\; , and so the Frobenius notation of this example is $(2,1|3,1)$.

\end{example}
\subsection{$\sigma$-tableau}
For a given permutation $\sigma\in {\frak S}_N$ and a Young tableau 
$T=(t_{i,j})$ of shape $\lambda$, 
we call $T$ a $\sigma$-tableau if it satisfies the following conditions:

(I) The entries of $T$ are weakly increasing along the rows in $\alpha(\lambda)$.

(II) The entries of $T$ are strictly increasing down the columns in $\beta(\lambda)\cup D^M(\lambda)$.

(III) $t_{\sigma(i), \sigma(i)}\leq t_{i, i+1}$ for $i=1, 2, \ldots, d$ whenever $i+1\leq \lambda_i$. \\


For ${\pmb s}\in T({\lambda}, {\mathbb C})$ and $T\in T({\lambda}, {\mathbb R})$, let
$$Z_{\lambda}({\pmb s}, T):=\prod_{(i, j)\in D(\lambda)}\frac{1}{t_{ij}^{s_{ij}}}.
$$
We easily find 
$$\displaystyle{\zeta_{\lambda}({\pmb s})=\sum_{M\in SSYT_{\lambda}}Z_{\lambda}({\pmb s}, M)}
\quad \displaystyle{\zeta_{\lambda}({\pmb s}|{\pmb x})=\sum_{M\in SSYT_{\lambda}}Z_{\lambda}({\pmb s}, M+{\pmb x})}.
$$

\begin{lem}\label{sigmatab}
{\rm{(\cite[Lemma 5.2]{MN4})}}
Let $T=(t_{ij})$ be a $\sigma$-tableau of shape $\lambda=(p_1, \ldots, p_N | q_1, \ldots, q_N)$. 
Then, 
\begin{align*}
T_k(\sigma)=
  \begin{array}{|c|c|c|c|}
 \hline
 t_{\sigma(k),\sigma(k)} & t_{k,k+1} & \cdots & t_{k,k+p_k}\\
 \hline
 t_{\sigma(k),\sigma(k)+1}\\
 \cline{0-0}
 \vdots\\
 \cline{0-0}
 t_{\sigma(k),\sigma(k)+q_{\sigma(k)}}\\
 \cline{0-0}
  \end{array}
  \; 
\end{align*}
is in
${\rm SSYT}(p_k | q_{\sigma(k)})$ for $1\leq k\leq N$,
and  for ${\pmb s}\in W_{\lambda}^{\rm diag}$, 
$$
Z_{\lambda}({\pmb s}, T)=\prod_{k=1}^N Z({\pmb s}_{k,\sigma(k)}^F, T_k(\sigma)),
$$
where ${\pmb s}_{k,\sigma(k)}^F$ is the part of ${\pmb s}$ corresponding to $T_k(\sigma)$ in the same notation as in Theorem \ref{Giambelli0}.
\end{lem}
Let
$${\mathbb X}:=\{(\sigma, T) | \sigma\in {\frak S}_N, T : \sigma \mbox{-tableau of shape} \; \lambda\}.
$$
 We see that 
$\mathbb{X}={\mathbb X}^{\dagger}\cup {\mathbb X}^{\dagger\dagger}$, 
where ${\mathbb X}^{\dagger}:=\{(\1, T) | T\in {\rm SSYT}({\lambda})\}$ with
$\1$ being the identity permutation,
and ${\mathbb X}^{\dagger\dagger}={\mathbb X}\setminus{\mathbb X}^{\dagger}$.

\begin{lem}\label{cancellation}
{\rm{(\cite[Lemma 5.3]{MN4})}}
For any $(\sigma,T)\in {\mathbb X}^{\dagger\dagger}$, 
there exists another element $(\sigma',T')\in {\mathbb X}^{\dagger\dagger}$ such that
\begin{align}\label{daggerzero}
{\rm sgn}(\sigma)Z_{\lambda}({\pmb s},T)+{\rm sgn}(\sigma')Z_{\lambda}({\pmb s},T')=0.
\end{align}
Therefore, it holds that
\begin{align*}
\sum_{(\sigma, T)\in {\mathbb X}}{\rm sgn}(\sigma) Z_{\lambda}({\pmb s}, T)=\sum_{(\1, T)\in {\mathbb X}} {\rm sgn}(\1)Z_{\lambda}({\pmb s}, T).
\end{align*}

\end{lem}


\subsection{Giambelli formula for Schur-Hurwitz multiple zeta-functions}

Here we describe a new result, the Giambelli formula for 
$\zeta_{\lambda}({\pmb s} | {\pmb x})$, as an extension of Theorem \ref{Giambelli0}.
\begin{thm}
\label{Giambelli_dayo}
With the same notation as in Theorem \ref{Giambelli0},
we have
\begin{equation}\label{expij}
\zeta_{\lambda}({\pmb s}|{\pmb x}) = \det(\zeta_{i,j}({\pmb s}_{i,j}^F|{\pmb x}_{i,j}^F))_{1 \leq i,j \leq N}.
\end{equation}
 \end{thm}

\begin{proof}
The right-hand side of \eqref{expij} is
\begin{align}\label{firststep}
&\det(\zeta_{i,j}({\pmb s}_{i,j}^F|{\pmb x}_{i,j}^F))_{1\leq i,j\leq N}
=\sum_{\sigma\in {\frak S}_N}{\rm sgn}(\sigma)\prod_{k=1}^N \zeta_{k, \sigma(k)}
({\pmb s}_{k, \sigma(k)}^F|{\pmb x}_{k,\sigma(k)}^F)\\
&\quad =\sum_{\sigma\in {\frak S}_N}{\rm sgn}(\sigma)\prod_{k=1}^N \sum_{M_k(\sigma)\in 
SSYT(p_k | q_{\sigma(k)})}Z_{(p_k | q_{\sigma(k)})}({\pmb s}_{k,\sigma(k)}^F, M_k(\sigma)+{\pmb x}_{k,\sigma(k)}^F).\notag
\end{align}

Then from \eqref{firststep} and Lemma \ref{sigmatab} we find

\begin{align*}
\det(\zeta_{i,j}({\pmb s}_{i,j}^F|{\pmb x}_{i,j}^F))_{1\leq i,j\leq N}
&=\sum_{\sigma\in {\frak S}_d}{\rm sgn}(\sigma) \sum_{T : \sigma{\rm{-tableau}}}Z_{\lambda}({\pmb s}, T+{\pmb x})\\
&=\sum_{(\sigma, T)\in {\mathbb X}}{\rm sgn}(\sigma) Z_{\lambda}({\pmb s}, T+{\pmb x}).
\end{align*}

We denote the right-hand side of the above by $W_{\mathbb X}$.


If $T\in {\rm SSYT}(\lambda)$, then $T$ is obviously a ${\1}$-tableau.
However $T$ is never a $\sigma$-tableau for any $\sigma\neq \1$.
In fact, if $T$ is a $\sigma$-tableau, then 
$t_{\sigma(i),\sigma(i)}\leq t_{i,i+1}$ for any $i$.
If $\sigma\neq\1$, then there exists an $i$ for which $j=\sigma(i)>i$, so
$t_{j,j}\leq t_{i,i+1}$.   However, since $T$ is SSYT, it should be
$t_{i,i+1}\leq t_{j-1,j}<t_{j,j}$, which is a contradiction.

Since
\begin{align*}
W_{{\mathbb X}^{\dagger}}:&=
\displaystyle{\sum_{({\1}, T)\in {\mathbb X}^{\dagger}}
{\rm sgn}(\1)Z_{\lambda}({\pmb s}, T+{\pmb x})}\\
&=\sum_{T\in {\rm SSYT}({\lambda})}Z_{\lambda}({\pmb s}, T+{\pmb x})=\zeta_{\lambda}({\pmb s} | {\pmb x}),
\end{align*}
we will show that $W_{\mathbb X}=W_{{\mathbb X}^{\dagger}}$. In other words,
it is enough to show that 
\begin{align}\label{daggerdagger}
W_{{\mathbb X}^{\dagger\dagger}}:=
\sum_{(\sigma, T)\in {\mathbb X}^{\dagger\dagger}}{\rm sgn}(\sigma)Z_{\lambda}({\pmb s}, T+{\pmb x})=0.
\end{align}
Note that $T$ is not SSYT for any element 
$(\sigma,T)\in {\mathbb X}^{\dagger\dagger}$.
We now consider the case of $T+{\pmb x}$ and $T'+{\pmb x}$ in stead of $T$ and $T'$ on the left-hand side of \eqref{daggerzero}, respectively. If ${\pmb x}\in T^{diag}(\lambda, {\mathbb R}_{\geq 0})$, the same 
element is added to each $(i, j)$-element of $T$ and $T'$. Therefore, the condition has no effect, and the Hurwiz version of Lemma \ref{cancellation} holds.
This leads to \eqref{daggerdagger}, and also the theorem.

\end{proof}

\subsection{Giambelli formula for Schur-Hurwitz multiple zeta-functions of skew type}

Recall the definition of Schur multiple zeta-functions of skew type.
For a skew Young diagram $\theta$, the Schur multiple zeta-function attached to $\theta$, which we denote by
$\zeta_{\theta}(\pmb{s})$, is defined similarly as \eqref{1-1}, just $\lambda$
replaced by $\theta$.
We also use the notation $T(\theta,X)$, $C(\theta)$ analogously to
$T(\lambda,X)$, $C(\lambda)$, respectively.
For any (skew) diagram $\theta$,
by $\theta^{\#}$ we denote the transpose of $\theta$ with respect to the anti-diagonal.
%
Then the authors obtained the Giambelli-type formula
for Schur multiple zeta-functions of skew type as follows.

\begin{thm}\label{Giambelli2}{\rm{(\cite[Theorem 3.2]{MN4})}}
Let  $\lambda$ be a partition whose
Frobenius notation is $\lambda=(p_1, \cdots , p_N | q_1, \cdots, q_N)$.
Let $\boldsymbol{\gamma} \in T(\lambda,\mathbb{N})$, and
assume $\boldsymbol{\gamma}^{\#}\in I_{\lambda^{\#}}$.
Then we have
\begin{align}\label{Giambelli2-formula}
\zeta_{\lambda^{\#}}(\boldsymbol{\gamma}^{\#})=
\det(\zeta_{(p_i,1^{q_j})^{\#}}(\gamma_{ij}^{F\#}))_{1\leq i,j\leq N}.
\end{align}
\end{thm}

\begin{example}
For $\lambda=(4, 3, 3, 2)$ and
{\small
$$
{\pmb \gamma}=
\ytableausetup{boxsize=normal}
  \begin{ytableau}
  c_0 & c_1 & c_2 & c_{3}  \\
  c_{-1} & c_{0} & c_{1} \\
 c_{-2}    & c_{-1} & c_{0}  \\ 
 c_{-3}    & c_{-2}  
  \end{ytableau},
$$ applying Theorem \ref{Giambelli2} gives
\begin{align*}
& \zeta_{\lambda^{\#}}\left(
\ytableausetup{boxsize=normal}
  \begin{ytableau}
   \none & \none & \none & c_{3}  \\
   \none & c_{0} & c_{1} &  c_{2} \\
 c_{-2}    & c_{-1} & c_{0} & c_{1}  \\ 
 c_{-3}    & c_{-2} & c_{-1} & c_{0} 
  \end{ytableau}\right)\\
&  =
  \left|
  \begin{matrix}
  \zeta_{(4,1^{3})^{\#}}\left(
\ytableausetup{boxsize=normal}
  \begin{ytableau}
   \none & \none & \none & c_{3}  \\
   \none &  \none &  \none &  c_{2} \\
  \none  &  \none &  \none & c_{1}  \\ 
 c_{-3}    & c_{-2} & c_{-1} & c_{0} 
  \end{ytableau}
\right)
& 
  \zeta_{(3,1^{3})^{\#}}\left(
\ytableausetup{boxsize=normal}
  \begin{ytableau}
 \none & \none & c_{3}  \\
  \none &  \none &  c_{2} \\
  \none &  \none & c_{1}  \\ 
 c_{-2} & c_{-1} & c_{0} 
  \end{ytableau}
\right)
&
  \zeta_{(1^{4})^{\#}}\left(
\ytableausetup{boxsize=normal}
  \begin{ytableau}
 c_{3}  \\
  c_{2} \\
 c_{1}  \\ 
 c_{0} 
  \end{ytableau}
\right)
\\
\\
  \zeta_{(4, 1)^{\#}}\left(
\ytableausetup{boxsize=normal}
  \begin{ytableau}
  \none  &  \none &  \none & c_{1}  \\ 
 c_{-3}    & c_{-2} & c_{-1} & c_{0} 
  \end{ytableau}
\right)
&
  \zeta_{(3,1)^{\#}}\left(
\ytableausetup{boxsize=normal}
  \begin{ytableau}
  \none &  \none & c_{1}  \\ 
 c_{-2} & c_{-1} & c_{0} 
  \end{ytableau}
\right)
&
  \zeta_{(1^{2})^{\#}}\left(
\ytableausetup{boxsize=normal}
  \begin{ytableau}
c_{1}  \\ 
 c_{0} 
  \end{ytableau}
\right)
\\
\\
  \zeta_{({4})^{\#}}\left(
\ytableausetup{boxsize=normal}
  \begin{ytableau}
 c_{-3}    & c_{-2} & c_{-1} & c_{0} 
  \end{ytableau}
\right)
&
  \zeta_{({3})^{\#}}\left(
\ytableausetup{boxsize=normal}
  \begin{ytableau}
 c_{-2} & c_{-1} & c_{0} 
  \end{ytableau}
\right)
&
  \zeta_{({1})^{\#}}\left(
\ytableausetup{boxsize=normal}
  \begin{ytableau}
 c_{0} 
  \end{ytableau}
  \right)
  \end{matrix}
  \right|
  .
\end{align*}
}
\end{example}

We now obtain an analogous result for the Hurwitz type as follows.
\begin{thm}\label{GiambelliH2}
Let  $\lambda$ be a partition whose
Frobenius notation is $\lambda=(p_1, \cdots , p_N | q_1, \cdots, q_N)$.
Let $\boldsymbol{\gamma} \in T(\lambda,\mathbb{N})$, and
assume $\boldsymbol{\gamma}^{\#}\in I_{\lambda^{\#}}$.
Then we have
\begin{align}\label{Giambelli2-formula}
\zeta_{\lambda^{\#}}(\boldsymbol{\gamma}^{\#} | {\pmb x}^{\#})=
\det(\zeta_{(p_i,1^{q_j})^{\#}}(\gamma_{ij}^{F\#} | x_{ij}^{\#}))_{1\leq i,j\leq N}.
\end{align}
\end{thm}

The proof of Theorem \ref{GiambelliH2} is identical with that of Theorem \ref{Giambelli2},
which uses the mapping from $\zeta_{\theta}$ to $\zeta_{\theta^{\#}}$ defined 
by the commutative diagram \eqref{commu} below.

Let ${\it Qsym}$ be the set of all quasi-symmetric functions, which is an algebra
with the harmonic product $*$.  It is known that ${\it Qsym}$ has a Hopf algebra
structure with coproduct $\Delta$ (see Hoffman \cite{Hof15}).
The antipode $\mathscr{S}$ is an automorphism of ${\it Qsym}$ and satisfies 
$\mathscr{S}^2={\rm id}$.


The Schur type quasi-symmetric function attached to $\theta$ is defined by
\begin{align}\label{3-1}
S_{\theta}(\boldsymbol{\gamma})=\sum_{M=(m_{ij}) \in {\rm SSYT}(\theta)}
\prod_{(i,j)\in D(\lambda)}t_{m_{ij}}^{\gamma_{ij}},
\end{align}
where $t_{m_{ij}}$ are variables, and
$\boldsymbol{\gamma}=(\gamma_{ij})\in T(\theta,\mathbb{N})$.

Let
$$
I_{\theta}=\{\boldsymbol{\gamma}=(\gamma_{ij})\in T(\theta,\mathbb{N})\mid
\gamma_{ij}\geq 2\;{\rm for \;all}\; (i,j)\in C(\theta)\}.
$$
Let $\rho\phi^{-1}$ be the mapping defined in \cite{Hof15} 
(see also \cite[Section 5]{NPY18}), which is a homomorphism (see \cite[p. 352]{Hof15}) and sends
$S_{\theta}(\boldsymbol{\gamma})$ to $\zeta_{\theta}(\boldsymbol{\gamma})$
when $\boldsymbol{\gamma}\in I_{\theta}$ (see \cite[Lemma 5.1]{NPY18}).
Then we define
$$\zeta_{\theta^{\#}}:=\rho \phi^{-1}(\mathscr{S}(S_{\theta}))$$ 
as in 
the following commutative diagram:
\begin{equation}\label{commu}
\begin{matrix}
S_{\theta} & \stackrel{{\mathscr{S}}}{\longrightarrow} & S_{\theta^{\#}}\\
\rotatebox{90}{$\longleftarrow$} & \rotatebox{180}{$\circlearrowright$} & \rotatebox{90}{$\longleftarrow$}\\
\zeta_{\theta} & \underset{{\rm{def}}}{\longrightarrow} & \zeta_{\theta^{\#}}
\end{matrix}
\end{equation}
With this commutative diagram, Theorem \ref{Giambelli_dayo} leads to Theorem \ref{GiambelliH2}.

\section{The sum of the products of $\zeta$ and $\zeta^{\star}$}\label{sec-5}

The Schur multiple zeta-function associated with $\lambda=(p+1, 1^{q})$ is called 
the {\it Schur multiple zeta-function of hook type}.
For $\lambda=(p+1, 1^{q})$, any ${\pmb s}\in W_{\lambda}$ is trivially belonging to
$W_{\lambda}^{\rm diag}$, hence we may write
\begin{equation}\label{atableau}
{\pmb s}=
\begin{array}{|c|c|c|c|}
\hline
s_{11} & s_{12} & \cdots & s_{1,p+1}\\
\hline
 s_{21} \\
\cline{1-1}
\vdots\\
\cline{1-1}
 s_{q+1,1}\\
 \cline{1-1}
\end{array}
=
\ytableausetup{boxsize=normal}  
\begin{ytableau}
\ytableausetup{centertableaux}
  z_{0} & z_{1} & \cdots & z_{p}\\
 z_{-1} \\
  \vdots \\
 z_{-q} 
\end{ytableau}\in W_{\lambda}^{\rm{diag}}\quad, 
\end{equation}
using the notation in Section \ref{sec-2}.
Similarly, any ${\pmb x}\in T(\lambda,\mathbb{R}_{\geq 0})$ for $\lambda=(p+1,1^q)$ is
trivially belonging to ${\pmb x}\in T^{\rm diag}(\lambda,\mathbb{R}_{\geq 0})$, so we
may use the notation ${\pmb y}=\{y_k\}$ introduced in Section \ref{sec-3}.

We now generalize Theorem \ref{thm3b} to the case of Schur-Hurwitz multiple 
zeta-functions, which can be stated as follows.
\begin{thm}\label{hook_zeta_zetastar}
For $\lambda=(p+1, 1^{q})$, 
we have
\begin{align}\label{hook1}
&\zeta_{\lambda}({\pmb s}|{\pmb x})=
\sum_{j=0}^{q}(-1)^j
\zeta^{\star}(z_{-j}, \ldots, z_{-1}, z_0, z_1, \ldots, z_p | y_{-j}, \ldots, y_{-1}, y_0, 
y_1, \ldots, y_p)\\
&\qquad\times\zeta(z_{-j-1}, \ldots, z_{-q} | y_{-j-1}, \ldots, y_{-q}),\notag
\end{align}
and
\begin{align}\label{hook2}
&\zeta_{\lambda}({\pmb s} | {\pmb x})=
\sum_{j=0}^{p}(-1)^j
\zeta(z_{j}, \ldots, z_{1}, z_0, z_{-1}, \ldots, z_{-q} | y_{j}, \ldots, y_{1}, y_0, y_{-1}, \ldots, y_{-q})\\
&\qquad\times\zeta^{\star}(z_{j+1}, \ldots, z_{p} | y_{j+1}, \ldots, y_{p}).\notag
\end{align}
\end{thm}

\begin{proof}
The proof runs along the same line as the proof of Theorem \ref{thm3} given in \cite{MN3},
so it is enough to show a brief outline.   We sketch the proof of \eqref{hook1}, which
is by induction on $q$.
The case $q=0$ is obvious, so we may assume that $q\geq 1$.
The definition of $\zeta_{\lambda}({\pmb s}|{\pmb x})$ 
for $\lambda=(p+1, 1^{q})$ is
\begin{align*}
&\zeta_{\lambda}({\pmb s}|{\pmb x})\\
&= \sum_{\substack{m_{11}\leq m_{12}\leq \ldots \leq m_{1,p+1}\\
m_{11}<m_{21}<\ldots <m_{q+1,1}}}
(m_{11}+x_{11})^{-z_{0}}(m_{12}+x_{12})^{-z_{1}}\\
&\qquad\times\cdots\times (m_{1, p+1}+x_{1,p+1})^{-z_{p}}
(m_{21}+x_{21})^{-z_{-1}}(m_{31}+x_{31})^{-z_{-2}}\\
&\qquad\times\cdots\times (m_{q+1,1}+x_{q+1,1})^{-z_{-q}}.
\end{align*}
We divide this multiple sum as
\begin{align*}
 \sum_{\substack{m_{11}\leq m_{12}\leq \ldots \leq m_{1,p+1}\\
m_{21}<\ldots <m_{q+1,1}}}
- \sum_{\substack{m_{11}\leq m_{12}\leq \ldots \leq m_{1,p+1}\\
m_{11}\geq m_{21}<\ldots <m_{q+1,1}}}.
\end{align*}
The first sum is obviously
$$
\zeta^{\star}(z_0,\ldots,z_p | y_0,\ldots,y_p)
\zeta(z_{-1},\ldots,z_{-q} | y_{-1},\ldots,y_{-q}).
$$
Applying the assumption of induction to the second sum, we obtain the result.
\end{proof}

\begin{thm}
\label{zeta_zetastar}
Let  $\lambda$ be a partition such that $\lambda=(p_1, \cdots , p_N | q_1, \cdots, q_N)$ 
in the Frobenius notation. 
Assume ${\pmb s}\in W_{\lambda}^{\rm{diag}}$ and
${\pmb x}\in T^{\rm diag}(\lambda,\mathbb{R}_{\geq 0})$.
Then we have
\begin{align*}
&\zeta_{\lambda}({\pmb s} | {\pmb x})= \sum_{\sigma\in {\frak S}_N} {\rm{sgn}}(\sigma) \sum_{j_1=0}^{q_1} \cdots \sum_{j_N=0}^{q_N}
(-1)^{j_1+\cdots +j_N}\\
& \times \zeta^{\star}(z_{-j_1}, \ldots, z_0, \ldots, z_{p_{\sigma(1)}} | y_{-j_1}, \ldots, 
y_0, \ldots, y_{p_{\sigma(1)}})\\
&\qquad\times
\zeta^{\star}(z_{-j_2}, \ldots, z_0, \ldots, z_{p_{\sigma(2)}} | y_{-j_2}, \ldots, y_0, \ldots, y_{p_{\sigma(2)}})\\
&\qquad\times\cdots\times
\zeta^{\star}(z_{-j_N}, \ldots, z_0, \ldots, z_{p_{\sigma(N)}} | y_{-j_N}, \ldots, y_0, \ldots, y_{p_{\sigma(N)}})\\
& \times \zeta(z_{-j_1-1}, \ldots, z_{-q_1} | y_{-j_1-1}, \ldots, y_{-q_1})
\zeta(z_{-j_2-1}, \ldots, z_{-q_2} | y_{-j_2-1}, \ldots, y_{-q_2})\\
&\qquad\times\cdots\times
\zeta(z_{-j_N-1}, \ldots, z_{-q_N} | y_{-j_N-1}, \ldots, y_{-q_N}).
\end{align*}
\end{thm}

\begin{proof}
The proof is by induction on $N$.     
The case $N=1$ is Theorem \ref{hook_zeta_zetastar}.
To prove the general case, we apply the Giambelli formula (Theorem \ref{Giambelli_dayo}),
which can be used because ${\pmb s}\in W_{\lambda}^{\rm{diag}}$
and ${\pmb x}\in T^{\rm diag}(\lambda,\mathbb{R}_{\geq 0})$.
The argument is just analogous to the proof of Theorem \ref{thm3b} given in
\cite{MN3}, so we omit the details.
\end{proof}

\section{Other expressions of Schur multiple zeta-functions of Hurwitz type}\label{sec-6}

%
%

\subsection{Zeta-functions of root systems}

Here we briefly recall the notion of {\it zeta-functions of root systems}
introduced by Komori, Matsumoto and Tsumura (\cite{KMT1}; see also \cite{KMT3}), which is
a representation-theoretic generalizations of the multiple zeta-functions of 
Euler-Zagier type \eqref{EulerZagier}.
Zeta-functions of root systems can be defined for any (reduced) root systems, but
here we only consider the case
when the root system is of type $A_r$.    In this case, the associated zeta-function 
is of the following simple form:
\begin{align}\label{rootzeta}
\zeta_r(\underline{\bf s}, A_r)=\sum_{m_1=1}^{\infty}\cdots \sum_{m_r=1}^{\infty}\prod_{1\leq i<j\leq r+1}(m_i+\cdots + m_{j-1})^{-s(i, j)},
\end{align}
where $s(i, j)$ is the variable corresponding to the root parametrized by $(i, j)$ ($1\leq i, j\leq r+1$, $i\not= j$) and
\begin{align*}
\underline{\bf s}=&(s(1, 2), s(2, 3), \ldots, s(r, r+1), s(1, 3), s(2, 4), \ldots, \\
& s(r-1, r+1), \ldots, s(1, r), s(2, r+1), s(1, r+1)).
\end{align*}

The modified zeta-function of root system was introduced by the authors in \cite{MN2}.
For $r> 0$ and $0\leq d \leq r$, they defined the modified zeta-function of the root system 
of type $A_r$ by
\begin{align}\label{Ahalfstar_def}
\zeta_{r, d}^{ \bullet}(\underline{\bf s},A_r)=&\underbrace{
\left(\sum_{m_1=0}^{\infty}\cdots \sum_{m_d=0}^{\infty}\right)'}_{d \text{ times}}
\underbrace{\sum_{m_{d+1}=1}^{\infty}\cdots \sum_{m_r=1}^{\infty}}_{r-d \text{ times}}\\
&\prod_{1\leq i<j\leq r+1}
(m_i+\cdots +m_{j-1})^{-s{(i,j)}}.\notag
\end{align}
The prime symbol in \eqref{Ahalfstar_def} means that 
the terms $(m_i+\cdots +m_{j-1})^{-s{(i,j)}}$, where $1\leq i<j\leq d+1$ and $m_i=\cdots =m_{j-1}=0$, are omitted.
If $r=0$, we put $\zeta_{r, d}^{ \bullet}(\underline{\bf s},A_r)=1$. 
We can say $\zeta_{r, d}^{\bullet}$ may be regarded as a "hybrid" version of zeta-functions and zeta-star-functions of root systems.
The authors also introduced
\begin{align}\label{AH_def}
\zeta_r^H(\underline{\bf s},x, A_r)=\sum_{m_1=1}^{\infty}\cdots \sum_{m_r=1}^{\infty}\prod_{1\leq i<j\leq r+1}
(x+m_i+\cdots +m_{j-1})^{-s{(i,j)}},
\end{align}
and
\begin{align}\label{Ahalfstarbullet_def}
&\zeta_{r, d}^{ \bullet, H}(\underline{\bf s}, x, A_r)\\
&=\underbrace{
\left(\sum_{m_1=0}^{\infty}\cdots \sum_{m_d=0}^{\infty}\right)}_{d \text{ times}}
\underbrace{\sum_{m_{d+1}=1}^{\infty}\cdots \sum_{m_r=1}^{\infty}}_{r-d \text{ times}}\prod_{1\leq i<j\leq r+1}
(x+m_i+\cdots +m_{j-1})^{-s{(i,j)}}, \notag
\end{align}
where $r>0$ and $x>0$. 

We notice that if $s(i,j)=0$ for all pairs $(i,j)$ with $i\geq 2$, then \eqref{rootzeta} 
reduces to the Euler-Zagier multiple zeta-function
$$
\zeta_{EZ,r}(s(1,2), s(1,3),\ldots,s(1,r+1)).
$$
Similarly we find that
\begin{align}
&\zeta^{\bullet, H}_{p_k, p_k}({\pmb s}_{+(k)}, m_{\sigma(k)\sigma(k)}, A_{p_k})
=\zeta_{EZ,p_k}^{\star\star}(z_1,\ldots,z_{p_k} | m_{\sigma(k)\sigma(k)},\ldots,
m_{\sigma(k)\sigma(k)}),\label{root_EZ1}\\
&\zeta_{q_j}^{H}({\pmb s}_{-(j)}, m_{jj}, A_{q_j})
=\zeta_{EZ,q_j}(z_{-1},\ldots,z_{-q_j} | m_{jj},\ldots,m_{jj}),\label{root_EZ2}
\end{align}
where ${\pmb s}_{+(k)}$ (resp. ${\pmb s}_{-(j)}$) is defined by
$s(1,\ell)=z_{\ell}$ ($1\leq \ell \leq p_k$)
(resp. $s(1,\ell)=z_{-\ell}$ ($1\leq \ell \leq q_j$)) and 
$s(i,j)=0$ for all pairs $(i,j)$ with $i\geq 2$.

In \cite{MN2}, the authors proved certain expressions of Schur multiple zeta-funcions
of anti-hook type in terms of the above (modified) zeta-functions of root systems.
Then in \cite{MN3}, the authors studied the case of content-parametrized Schur multiple
zeta-functions.    Theorem \ref{mainthm3} in Section $2$ is one of the main results in
\cite{MN3}.
The motivation in \cite{MN3} is to develop an analogue of the
arguments in \cite{MN2}, and hence in \cite{MN3}, Theorem \ref{mainthm3} is stated
in terms of zeta-functions of root systems; that is, $\zeta_{EZ,p_k}^{\star\star}$
and $\zeta_{EZ,q_j}$ on the right-hand side of \eqref{Dirichlet} are replaced by
the left-hand sides of \eqref{root_EZ1} and \eqref{root_EZ2}, respectively.
(Note that the definitions of ${\pmb s}_{+(k)}$, ${\pmb s}_{-(j)}$ given in 
\cite[p. 17]{MN3} are not accurate; 0-terms are lacking, so those should be amended 
as above.)

%

\subsection{The sum of the products of $\zeta$ and $\zeta^{\star\star}$}

The aim of this subsection is to generalize Theorem \ref{mainthm3} to the case of 
Schur-Hurwitz multiple zeta-functions.

We first consider the Schur-Hurwitz multiple zeta-function of hook type, corresponding to
$\lambda=(p+1,1^q)$.
From the definition of semi-standard Young tableaux, we see that the runnning indices for 
the definition of $\zeta_{\lambda}({\pmb s} | {\pmb x})$ satisfy 
$1\leq m_{11}\leq m_{12}\leq \cdots \leq m_{1, p+1}$ and
$1\leq m_{11}< m_{21}< \cdots < m_{q+1,1}$. 
Therefore, setting $m_{12}=m_{11}+a_{1}$ ($a_{1}\geq 0$), 
$m_{13}=m_{11}+a_{1}+a_{2}$ ($a_{1}, a_{2}\geq 0$),
$\cdots$,
\begin{equation}\label{k0ina}
m_{1, p+1}=m_{11}+a_{1}+a_{2}+\cdots +a_{p} \quad (a_{i}\geq 0),
\end{equation}
and  $m_{21}=m_{11}+b_{1}$ ($b_{1}\geq 1$), 
$m_{31}=m_{11}+b_{1}+b_{2}$ ($b_{1}, b_{2}\geq 1$),
$\cdots$,
$$m_{q+1,1}=m_{11}+b_{1}+b_{2}+\cdots +b_{q} \quad (b_{j}\geq 1), $$
we can write the Schur-Hurwitz multiple zeta-function of hook type as 
\begin{align}\label{thm4_hooktype}
&\zeta_{\lambda}({\pmb s} | {\pmb x})
= \sum_{\substack{m_{11}\geq 1\\
a_i\geq 0 (1\leq i \leq p)\\
b_j\geq 1(0\leq j\leq q)}}(m_{11}+x_{11})^{-z_{0}}(m_{11}+a_1+x_{12})^{-z_{1}}\\
&\quad\times\cdots\times (m_{11}+a_1+\cdots + a_{p}+x_{1,p+1})^{-z_{p}}\notag\\
&\quad\times (m_{11}+b_{1}+x_{21})^{-z_{-1}}
(m_{11}+b_{1}+ b_2+x_{31})^{-z_{-2}}\notag\\
&\quad\times\cdots\times 
 (m_{11}+b_{1}+b_{2}+\cdots + b_{q}+x_{q+1,1})^{-z_{-q}}
 \notag\\
&= \sum_{m_{11}\geq 1} (m_{11}+x_{11})^{-z_{0}}\zeta_{EZ,p}^{\star\star}(z_1,\ldots,z_p | 
m_{11}+x_{12},\ldots,m_{11}+x_{1,p+1})\notag\\
&\quad\times \zeta_{EZ,q}(z_{-1},\ldots,z_{-q} |  m_{11}+x_{21}, \ldots,m_{11}+x_{q+1,1}).
\notag\\
&= \sum_{m_{11}\geq 1} (m_{11}+y_{0})^{-z_{0}}\zeta_{EZ,p}^{\star\star}(z_1,\ldots,z_p | 
m_{11}+y_1,\ldots,m_{11}+y_p)\notag\\
&\quad\times \zeta_{EZ,q}(z_{-1},\ldots,z_{-q} |  m_{11}+y_{-1}, \ldots,m_{11}+y_{-q}).
\notag\ 
\end{align}
This is a generalization of \cite[(3.10)]{MN3}.

Now we are ready to show the following theorem, which gives a generalization of
Theorem \ref{mainthm3}.

\begin{thm}
\label{mainthm4}
Let  $\lambda$ be a partition such that
$\lambda=(p_1, \cdots , p_N | q_1, \cdots, q_N)$ in the Frobenius notation.
Assume ${\pmb s}\in W_{\lambda}^{\rm{diag}}$
and ${\pmb x}\in T^{\rm diag}(\lambda,\mathbb{R}_{\geq 0})$.
Then we have
\begin{align}\label{Dirichlet2}
&\zeta_{\lambda}({\pmb s})
= \sum_{\substack{m_{11}, m_{22}, \ldots, m_{NN}\geq 1}}
((m_{11}+y_0)\ldots (m_{NN}+y_0))^{-z_{0}}\\
&\times \sum_{\sigma\in {\frak S}_N} {\rm{sgn}}(\sigma)
\prod_{k=1}^N
\zeta_{EZ,p_k}^{\star\star}(z_1,\ldots,z_{p_k} | m_{\sigma(k)\sigma(k)}+y_1,\ldots,
m_{\sigma(k)\sigma(k)}+y_{p_k})\notag\\
&\times\prod_{j=1}^N 
\zeta_{EZ,q_j}(z_{-1},\ldots,z_{-q_j} | m_{jj}+y_{-1},\ldots,m_{jj}+y_{-q_j}).\notag
\end{align}
\end{thm}

\begin{proof}
The proof is induction on $N$.    The case $N=1$ is \eqref{thm4_hooktype}.
The general case can be deduced by using the Giambelli formula 
(Theorem \ref{Giambelli_dayo}), following the same way as in the proof of 
Theorem \ref{mainthm3} given in \cite{MN3}.
\end{proof}
%

\begin{remark}
As mentioned in \cite[Remark 1, Remark 7]{MN3}, Theorem \ref{mainthm3} suggests
that Schur multiple zeta-functions have some connections with Weyl group multiple Dirichlet
series studied by Bump, Goldfeld and others (see \cite{Bump}).   We can observe the same
similarity with Weyl group multiple Dirichlet series in Theorem \ref{mainthm4}.
\end{remark}

\section{Derivative formulas}\label{sec-7}

In the previous sections we obtained several identities  involving Schur-Hurwitz
multiple zeta-functions, which include real parameters $x_{ij}$ or $y_k$.
By differentiating those identities with respect to those parameters, we may expect
to arrive at certain new type of identities.
We conclude the present paper with presenting just one example of such derivative
formulas.    

\begin{remark}
The idea of differentiating some formulas among double polylogarithms
with respect to a shifting parameter to obtain new identities was used in
\cite{MT}.
\end{remark}

We differentiate the both sides of \eqref{hook1} with respect to $y_\ell$, 
$0\leq \ell\leq p$.
On the left-hand side, we see that
\begin{align*}
&\frac{\partial}{\partial y_{\ell}}\zeta_{\lambda}({\pmb s} | {\pmb x})
=\sum_{M\in{\rm SSYT}(\lambda)}\prod_{\substack{(i,j)\in D(\lambda) \\ j-i\neq\ell}}
(m_{ij}+x_{ij})^{-s_{ij}}\frac{\partial}{\partial y_{\ell}}
\Biggl(\prod_{\substack{(i,j)\in D(\lambda) \\ j-i=\ell}}(m_{ij}+y_{\ell})^{-z_{\ell}}
\Biggr)\\
&=-z_{\ell}\sum_{M\in{\rm SSYT}(\lambda)}\prod_{\substack{(i,j)\in D(\lambda) \\ j-i\neq\ell}}
(m_{ij}+x_{ij})^{-s_{ij}}\\
&\qquad\qquad\times\sum_{\substack{(i_1,j_1)\in D(\lambda) \\ j_1-i_1=\ell}}
(m_{i_1 j_1}+y_{\ell})^{-z_{\ell}-1}
\prod_{\substack{(i,j)\in D(\lambda) \\ j-i=\ell \\ (i,j)\neq (i_1,j_1)}}
(m_{ij}+y_{\ell})^{-z_{\ell}}\\
&= -z_{\ell}\sum_{\substack{(i_1,j_1)\in D(\lambda) \\ j_1-i_1=\ell}}
\zeta_{\lambda}({\pmb s}_{1}(i_1,j_1) | {\pmb x}),
\end{align*}
where ${\pmb s}_{a}(i_1,j_1)$ is almost the same as ${\pmb s}$ but $s_{i_1 ,j_1}$ is
replaced by $s_{i_1, j_1}+a$ ($a=1,2$).
It is easier to differentiate the right-hand side of \eqref{hook1}.   The result is the
first identity in the following theorem.    The second identity can be shown by
differentiating one more time.
\begin{thm}
For $\lambda=(p+1,1^q)$, we have
\begin{align}\label{7-2-1}
&\sum_{\substack{(i_1,j_1)\in D(\lambda) \\ j_1-i_1=\ell}}
\zeta_{\lambda}({\pmb s}_{1}(i_1,j_1) | {\pmb x})\\
&=\sum_{j=0}^{q}(-1)^j
\zeta^{\star}(z_{-j}, \ldots, z_{-1}, z_0,\ldots,z_{\ell}+1,\ldots, z_p | y_{-j}, \ldots, y_{-1}, y_0, 
y_1, \ldots, y_p)\notag\\
&\qquad\times\zeta(z_{-j-1}, \ldots, z_{-q} | y_{-j-1}, \ldots, y_{-q})\notag
\end{align}
and
\begin{align}\label{7-2-2}
&\sum_{\substack{(i_1,j_1)\in D(\lambda) \\ j_1-i_1=\ell}}
\zeta_{\lambda}({\pmb s}_{2}(i_1,j_1) | {\pmb x})
+2\sum_{\substack{(i_1,j_1),(i_2,j_2)\in D(\lambda) \\
(i_1,j_1)\neq (i_2,j_2)\\ j_1-i_1=j_2-i_2=\ell}}
\zeta_{\lambda}({\pmb s}_{1}((i_1,j_1),(i_2,j_2)) | {\pmb x})
\\
&=\sum_{j=0}^{q}(-1)^j
\zeta^{\star}(z_{-j}, \ldots, z_{-1}, z_0,\ldots,z_{\ell}+2,\ldots, z_p | y_{-j}, \ldots, y_{-1}, y_0, 
y_1, \ldots, y_p)\notag\\
&\qquad\times\zeta(z_{-j-1}, \ldots, z_{-q} | y_{-j-1}, \ldots, y_{-q}),\notag
\end{align}
where ${\pmb s}_{1}((i_1,j_1),(i_2,j_2))$ is almost the same as ${\pmb s}$ but
$s_{i_1 ,j_1}$ and $s_{i_2,j_2}$ are
replaced by $s_{i_1, j_1}+1$ and $s_{i_2,j_2}+1$, respectively.
\end{thm}

The same type of identities can be deduced by differentiating more and more.
Similar discussion can be done by starting from \eqref{hook2} instead of \eqref{hook1}.
Other formulas proved in the previous sections can also be used as the starting point,
so we can obtain a lot of derivative formulas.
Moreover it is to be stressed that, putting $y_j=0$ for all $j$ after the
differentiation, we can show various identities among non-shifted multiple
zeta-functions of Schur type and of Euler-Zagier type.

 
%
\bigskip
\noindent
\textsc{Kohji Matsumoto}\\
Graduate School of Mathematics, \\
Nagoya University, \\
Furo-cho, Chikusa-ku, Nagoya, 464-8602, Japan \\
 \texttt{kohjimat@math.nagoya-u.ac.jp}\\
and\\
Center for General Education, \\
Aichi Institute of Technology, \\
1247 Yachigusa, Yakusa-cho, Toyota, 470-0392, Japan \\

\medskip

\noindent
\textsc{Maki Nakasuji}\\
Department of Information and Communication Science, Faculty of Science, \\
 Sophia University, \\
 7-1 Kioi-cho, Chiyoda-ku, Tokyo, 102-8554, Japan \\
 \texttt{nakasuji@sophia.ac.jp}\\
and\\
Mathematical Institute, \\
Tohoku University, \\
6-3 Aoba, Aramaki, Aoba-ku, Sendai, 980-8578, Japan \\
  
\end{document}